\documentclass[11pt]{article} 

\usepackage[utf8]{inputenc} 

\usepackage{amsmath, amsthm, amsfonts, amssymb}
\usepackage{geometry} 
\usepackage{colonequals}
\usepackage{graphicx} 
\usepackage{tikz} \usetikzlibrary {positioning}

\usepackage{booktabs} 
\usepackage{array} 
\usepackage{paralist} 
\usepackage{verbatim} 
\usepackage{subfig} 
\usepackage{fancyhdr} 
\newtheorem{theorem}{\sc Theorem}[section]
\newtheorem{lemma}[theorem]{\sc Lemma}
\newtheorem{prop}[theorem]{\sc Proposition}
\newtheorem{definition}{Definition}[section]

\newtheorem{corollary}[theorem]{\sc Corollary}

\newtheorem{conj}{\sc Conjecture}[section]
\newtheorem{remark}{Remark}
\renewcommand{\hom}{\mbox{\rm Hom}}
\newcommand{\ind}{\mbox{\rm ind}}
\newcommand{\coind}{\mbox{\rm coind}}
\newcommand{\conn}{\mbox{\rm conn}}
\newcommand{\dash}{\mbox{---}}
\newcommand{\cq}{\colonequals}

\def\Z{\mathbb{Z}}

\title{Warmth and connectivity of \\
neighborhood complexes of graphs}
\author{Anton Dochtermann and Ragnar Freij-Hollanti}

\begin{document}

\maketitle
\begin{abstract}
In this paper we study a pair of numerical parameters associated to a graph $G$.   One the one hand, one can construct $\hom(K_2, G)$, a space of homomorphisms from a edge $K_2$ into $G$ and study its (topological) connectivity.  This approach dates back to the neighborhood complexes introduced by Lov\'asz in his proof of the Kneser conjecture.   In another direction Brightwell and Winkler introduced a graph parameter called the warmth $\zeta(G)$ of a graph $G$, based on asymptotic behavior of $d$-branching walks in $G$ and inspired by constructions in statistical physics.   Both the warmth of $G$ and the connectivity of $\hom(K_2,G)$ provide lower bounds on the chromatic number of $G$.

Here we seek to relate these two constructions, and in particular we provide evidence for the conjecture that the warmth of a graph $G$ is always less than three plus the connectivity of $\hom(K_2, G)$.  We succeed in establishing a first nontrivial case of the conjecture, by showing that $\zeta(G) \leq 3$ if $\hom(K_2,G)$ has an infinite first homology group. We also calculate warmth for a family of `twisted toroidal' graphs that are important extremal examples in the context of $\hom$ complexes.  Finally we show that $\zeta(G) \leq n-1$ if a graph $G$ does not have the complete bipartite graph $K_{a,b}$ for $a+b=n$.  This provides an analogue for a similar result in the context of $\hom$ complexes.  

\end{abstract}

\section{Introduction}
\label{sec:in}

Suppose $G$ is a graph with no multiple edges. In recent years a pair of numerical invariants associated to $G$ have been introduced in seemingly independent contexts. On the one hand, one can construct a \emph{space} of homomorphisms from a edge $K_2$ into $G$ and study various notions of topological \emph{connectivity}. This construction dates back to the neighborhood complexes $N(G)$ introduced by Lov\'asz in his proof of the Kneser conjecture, and was further developed by Babson and Kozlov.   In modern treatments we recover $N(G) \simeq \hom(K_2, G)$ as a space of homomorphisms from the edge $K_2$ into the graph $G$ (homomorphisms from an edge), an example of the more general $\hom$-complexes $\hom(T,G)$ of homomorphisms between two graphs $T$ and $G$ (see \cite{BabKoz}).  Precise definitions are given in Section~\ref{sec:Basics}, but the basic idea is that $\hom(K_2, G)$ is a polyhedral complex with 0-cells given by all directed edges of $G$, with higher dimensional cells given by directed complete bipartite graphs. 

In another direction Brightwell and Winkler studied notions of `long range action' of graph homomorphisms, motivated by constructions in statistical physics. They introduced a graph parameter called the \emph{warmth} $\zeta(G)$ of a graph $G$, a measure of the asymptotic behavior of $d$-branching walks in $G$.   The idea is a generalization of the following observation: if $B$ is a \emph{bipartite} graph then we can restrict the possibilities of the initial position of a random walk (thought of as a map from the one-branching tree $T^1$ to $B$) if we know where the walk is at the $n$th step, regardless of how large $n$ is. The warmth $\zeta(G)$ quantifies, in a way that will be made precise in Section~\ref{sec:Basics}, how large $d$ needs to be for the same to be true for maps $T^d\to G$. 

It turns out the both the warmth of $G$ and the connectivity of the edge space of $G$ provide lower bounds on the chromatic number, by definition the fewest number of colors needed to color the vertices of $G$ in such a way that adjacent vertices receive distinct colors. While upper bounds on chromatic number are in some sense straightforward (just `write down a coloring') finding lower bounds often requires methods from diverse branches of mathematics.   
In this language, a main result of \cite{Lov} is that for any graph $G$ we have
\[ \chi(G) \geq \conn(\hom(K_2,G)) + 3.\]
\noindent
Here we use the convention that $\conn(X) = -1$ if $X$ is nonempty and disconnected.  In more recent work \cite{BabKozproof}. On the statistical physics side, Brightwell and Winkler \cite{BriWin} show that warmth provides a lower bound on chromatic number: For any graph $G$ we have
\[ \chi(G) \geq \zeta(G). \]



As both warmth $\zeta(G)$ and connectivity of $N(G)$ provide lower bounds on chromatic number, the natural question that arises is whether the two parameters themselves are related.  We note that both are defined in terms of homomorphisms of certain graphs \emph{into} $G$ (which is why it is somewhat surprising that they relate to $\chi(G)$, which is defined in terms of homomorphisms \emph{from} $G$).  In particular they both behave well with respect to categorical graph products.  Both invariants also provide tight bounds on the chromatic number of complete graphs $K_n$, with $\zeta(K_n) = n$ and $\conn(\hom(K_2,K_n)) = n-3$. The space $\hom(K_2,K_n)$ can in fact be realized as the boundary of an $(n-1)$-dimensional convex polytope.



A small example of a graph $G$ satisfying $\zeta(G) < \conn(N(G))+3$ is given by the \emph{Gr\"otszch graph} on 11 nodes and 20 edges~\cite{BriWin}, see Figure \ref{Fig:grotszch}.

\medskip

\begin{figure}[ht]

\begin{center}
  \includegraphics[scale = 0.35]{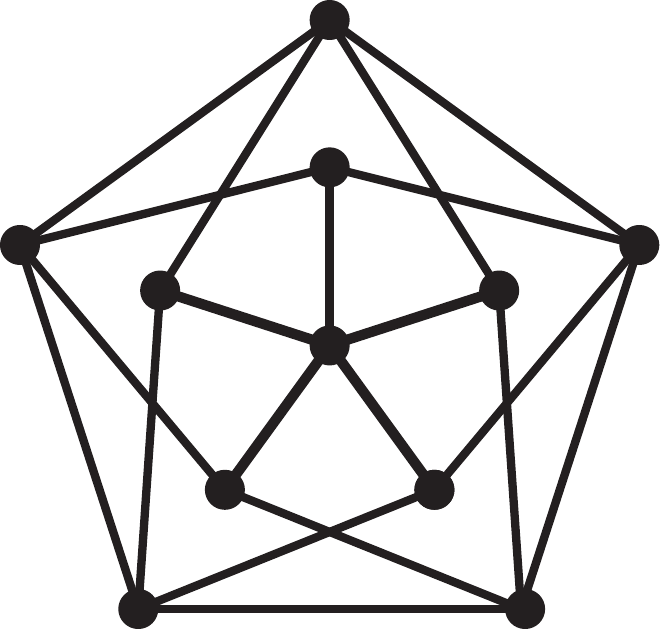}

    \caption{The Gr\"otszch graph $G$, with $\zeta(G) = 3$ and $\conn(N(G)) = 1$.}
\label{Fig:grotszch}

\end{center}
\end{figure}
Indeed, the Gr\"otszch graph is a special case of the Mycielski construction $G\mapsto M(G)$. The Mycielski construction always increases the chromatic number of a graph $G$ by 1, and up to homotopy corresponds to suspending the neighborhood complex $N(G)$~\cite{GyaJenSti} .  Hence it preserves tightness of the bound $\conn(N(G))+3\leq\chi(G)$. However, it may fail to increase warmth: the Gr\"otszch graph $M(C_5)$ has warmth \[\zeta(M(C_5))=\zeta(C_5)=3,\] as shown in~\cite{BriWin}. Thus, the warmth of a graph $G$ can be smaller than three plus the connectivity of $N(G)$.  

In fact the gap between warmth and connectivity of the neighborhood complex can be arbitrarily large.  As discussed in \cite{FadKah}, the family of Kneser graph $K_{3k, k}$ has neighborhood complex connectivity $k-1$ and constant warmth 3.   Furthermore, if we allow loops on our vertices (which of course takes us out of the world of finite chromatic number) we see that there can even be an infinite gap between the values of these two parameters.  In \cite{BriWin} it is shown that a graph $G$ has $\zeta(G) = \infty$ if and only if $G$ is `dismantlable', meaning that it can be folded down to a single looped vertex (see Section \ref{sec:Bipartite} for more details).  There exist graphs $G$ with the property that $N(G)$ is contractible and yet $G$ is not dismantlable, so in this case we have $\conn(N(G)) = \infty$ and yet $\zeta(G)$ finite. 

This leads to the question whether warmth is \emph{always} less than three plus the connectivity of the neighborhood complex. Indeed the following conjecture, first published in \cite{FadKah}, will motivate the rest of our work in this paper.

\begin{conj}\label{mainconj}
For any finite graph $G$ we have
 \[\zeta(G)\leq \conn(\hom(K_2,G))+3.\]
\end{conj}

Conjecture \ref{mainconj} suggests a perhaps surprising relationship between the long range action of infinite $d$-branching trees and the homotopy classes of maps from a $d$-sphere into the edge space of a graph. It should be noted that if $T$ is any \emph{finite} tree, it can be shown that $\hom(T,G)$ is homotopy equivalent to the space $\hom(K_2, G)$ (see \cite{BabKoz}).  It should be noted that a first case of \ref{mainconj} with \[ \zeta(G) = 2 \Leftrightarrow (\conn(\hom(K_2,G)) = -1) \]
follows from the fact that a (nonempty, connected) graph $G$ has disconnected $N(G)$ if and only if $G$ is bipartite; which is the case if and only if $\zeta(G) = 2$.  Our first main result establishes a first nontrivial case of Conjecture \ref{mainconj}.

\newtheorem*{thm:fundgroup}{\sc Theorem \ref{thm:fundgroup}}
\begin{thm:fundgroup}
Suppose $G$ is a graph such that the first homology group $H_1(\hom(K_2,G))$ contains an infinite cyclic subgroup. Then $\zeta(G)\leq 3$.
\end{thm:fundgroup}

In the search for a counterexample to Conjecture~\ref{mainconj}, a natural graph to consider would be one where the connectivity bound fails give the true value of $\chi(G)$.  In the modern treatment~\cite{BabKoz, DocSch} of the original results from \cite{Lov} one can see that it is in fact the ${\mathbb Z}_2$-\emph{index} of the neighborhood complex that provides the weakest bound in the hierarchy of topological bounds on chromatic number.   In \cite{DocSch} a class of graphs $T_{k,m}$ called `twisted toroidal graphs' were constructed recursively, with the property that the connectivity of their neighborhood complexes are fixed while the ${\mathbb Z}_2$-index of the same complex increases.  The main result in Section \ref{sec:Twisted} is that the warmth of these graphs is similarly held constant as one iterates the construction, providing a family of graphs where the warmth bound is arbitrarily far from chromatic number. 

\newtheorem*{thm:twisted}{\sc Theorem~\ref{thm:twisted}}
\begin{thm:twisted}
For all $m\geq5$ and $k \geq 1$ the twisted toroidal graph $T_{k,m}$ has warmth $3$.
\end{thm:twisted}

Finally, we turn our attention to bipartite subgraphs, and the effect they have on warmth and connectivity of neighborhood complexes. In \cite{CLSW} the authors show that if a graph $G$ does not have a subgraph isomorphic to $K_{a,b}$ for some $a+b = n$, then $N(G)$ deformation retracts onto a complex of dimension $n-3$ (see also \cite{Kah} for a proof of this fact). In particular, if $\chi(G)$ is moreover finite, then $N(G)$ has connectivity $\leq n-4$. The main result in Section~\ref{sec:Bipartite} is an analogous result for warmth, further supporting Conjecture~\ref{mainconj}.

\newtheorem*{corr:bipartite}{\sc Theorem \ref{corr:bipartite}}
\begin{corr:bipartite}
Suppose $a$ and $b$ positive integers with $a+b \geq 3$.  If a graph $G$ does not contain any subgraph isomorphic to $K_{a,b}$, then $\zeta(G)\leq a+b-1$.
\end{corr:bipartite}

The rest of the paper is organized as follows.   In Section \ref{sec:Basics} we review the basic objects involved in our study, including precise definitions and characterizations of the neighborhood complex and warmth.  In Section \ref{sec:Fundamental} we establish the first nontrivial case of Conjecture \ref{mainconj}, Section \ref{sec:Twisted} is devoted to warmth of twisted toroidal graphs, and in Section \ref{sec:Bipartite} we prove that the warmth of a graph $G$ depends on whether $G$ has all possible complete bipartite subgraphs of a given size.  We end in Section \ref{sec:further} with some further questions and discussions.

\medskip
\noindent
{\bf Acknowledgments.}  The authors wish to thank Matthew Kahle for fruitful discussions, as well as an anonymous referee for helpful suggestions.

\section{Basics}\label{sec:Basics}

In this section we collect some basic definitions and results that will be needed for our work.  Most of our terminology follows \cite{BriWin} and \cite{BabKoz}.

A graph $G$ consists of a set of vertices (or nodes) $V(G)$ along with a collection of undirected edges $E(G)$.  All graphs in this paper are assumed to have no multiple edges, but may in some cases have loops.  We will always insist that $G$ has at least one edge and will typically assume that $G$ is connected, unless specified otherwise. The \emph{neighborhood} of a node $v\in V(G)$ is defined as $N(v)=\{u\in V(G) : vu\in E(G)\}$, and the neighborhood of a set $A\subseteq V(G)$ is defined as $N(A)=\cup_{v\in A}N(v)$.  A \emph{graph homomorphism} (or graph map) $f:G \rightarrow H$ is a mapping of the vertex set $f:V(G) \rightarrow V(H)$ that preserves adjacency, so that if $\{v,w\} \in E(G)$ then $\{f(v),f(w)\} \in E(H)$.

\begin{definition}
The chromatic number $\chi(G)$ of a graph $G$ is defined as
\[\chi(G) = \min\{n : \textrm{There exists a graph homomorphism $G \rightarrow K_n$} \}.\]
\end{definition}

Here we use $K_n$ to denote the complete graph consisting of $n$ vertices and all possible (non loop) edges.  We let $T^d$ denote the rooted infinite $d$-branching tree, an infinite connected graph consisting of a distinguished node $r$ with degree $d$, with all other vertices having degree $d+1$ (see Figure \ref{Fig:T2}).  In particular, $T^1$ is a rooted infinite path.  

\begin{figure}[h]
\begin{center}
  \includegraphics[scale = 0.35]{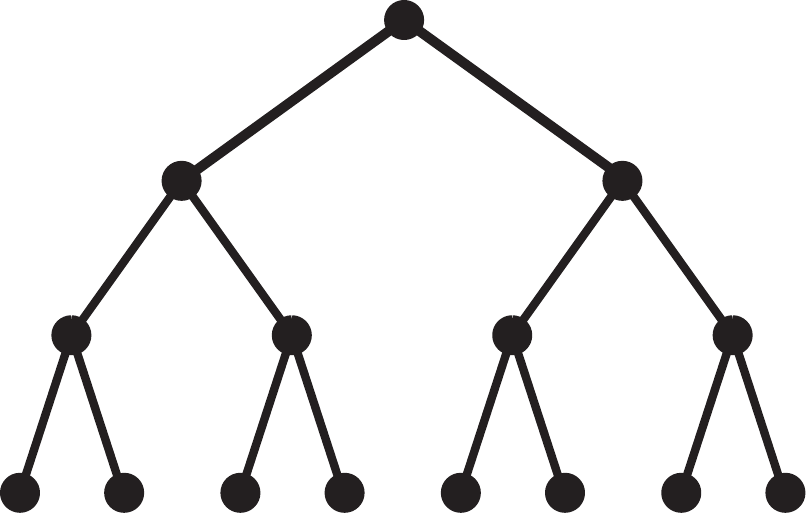}

    \caption{(A piece of) the graph $T_2$.}
\label{Fig:T2}
\end{center}
\end{figure}

\noindent

\subsection{Warmth}
We first address the warmth of a graph as discussed in the introduction.  We begin with some basic definitions form \cite{BriWin}.

\begin{definition}
A graph map $\varphi:T^d \rightarrow G$ is \emph{cold} if there exists a node $a \in V(G)$ such that for any integer $k$, there does not exist a graph map $\psi:T^d \rightarrow G$ such that $\psi$ agrees with $\varphi$ on the vertices of $T^d$ at distance $k$ from the root $r$ but has $\psi(r) = a$.  
\end{definition}

One can see that if $\varphi:T^d \rightarrow G$ is cold then there exists a cold map $\varphi^\prime:T^{d+1} \rightarrow G$, simply by extending $\varphi$ in an arbitrary way.  We are then lead to the following definition.

\begin{definition}
For a graph $G$ we define its \emph{warmth} $\zeta(G)$ according to
\[\zeta(G) = \min\{d: \textrm{There exists a cold map $\varphi: T^{d-1} \rightarrow G$}.\}\]
\end{definition}

In this context we will always assume that $G$ has at least one edge. In particular, the set of maps $T^d\to G$ is non-empty, so by definition $\zeta(G) \geq 2$.  There is a nice characterization of warmth from \cite{BriWin} that will often be more convenient to work with.   For this we need the following.

\begin{definition}
A nonempty family $\mathcal{A}$ of nonempty proper subsets of $V(G)$ is called \emph{$d$-stable} (in $G$) if, for every $A\in\mathcal{A}$ there exists a subfamily $\{A_i\}_{i=1}^d\subseteq\mathcal{A}$ of size $d$, such that
\[\bigcap_i N(A_i)=A.\]
\end{definition}

We then have the following characterization of warmth, again from \cite{BriWin}.
\begin{prop}
If $\zeta(G)$ denotes the warmth of a graph $G$, we have $\zeta(G) \leq d+1$ if and only if there exists a $d$-stable family in $G$.
\end{prop}

For example, the warmth of a disconnected graph is always $2$, a trivial 1-stable family consists of only one set (namely one of the components). Almost as trivially, a bipartite graph has warmth 2, as a 1-stable family is given by the two parts of the bipartition.

A slightly more delicate example is the fact that a graph $G$ with girth $g\geq 5$ always has $\zeta(G)\leq 3$. Indeed, let $a_0\cdots a_{g-1} a_0$ be a minimal cycle of $G$, and let \[\mathcal{A}=\{\{a_i\}: i=0,\dots , g-1\}.\] Then $\mathcal{A}$ is $2$-stable, as we have $N(a_{i-1})\cap N(a_{i+1})=\{a_i\}$ (counted modulo $g$). 

The main result regarding warmth from \cite{BriWin} is the following connection to chromatic number.
\begin{theorem}\label{thm:warmthbound}
For any graph $G$ we have
\[\zeta(G) \leq \chi(G).\]
\end{theorem}

\subsection{Hom and Neighborhood complexes}

In his proof of the Kneser conjecture \cite{Lov}, Lov\'asz introduced the so-called neighborhood complex $N(G)$ of a graph $G$ and showed that the topology of $N(G)$ provided a lower bound on the chromatic number of $G$.  It turns out that $N(G)$ is an example of a more general homomorphism complex $\hom(T,G)$ parametrizing graph homomorphisms from $T$ to $G$.  It can be shown that $N(G)$ is homotopy equivalent to $\hom(K_2,G)$, and it is the latter construction that we will need for our purposes \cite{BabKoz}.  For the following we use $\Delta_S$ to denote the simplex on vertex set $S$.

\begin{definition}
For a finite graph $G$, $\hom(K_2, G)$ is the subcomplex of $\Delta_{V(G)} \times \Delta_{V(G)}$ consisting of all faces $\sigma \times \tau$ with $\sigma, \tau \neq \emptyset$, $\sigma, \tau \subseteq V(G)$, and such that if $v \in \sigma,  w \in \tau$, then $\{v,w\} \in E(G)$.
\end{definition}

If $\sigma=\{s\}$ and $\tau=\{t\}$ are both singletons, we use the shorthand notation $(s,t)$ to denote the vertex $\sigma\times\tau$.  See Figure \ref{Fig:homcomplex}

\begin{figure}[htb]
\begin{center}
  \includegraphics[scale = 0.175]{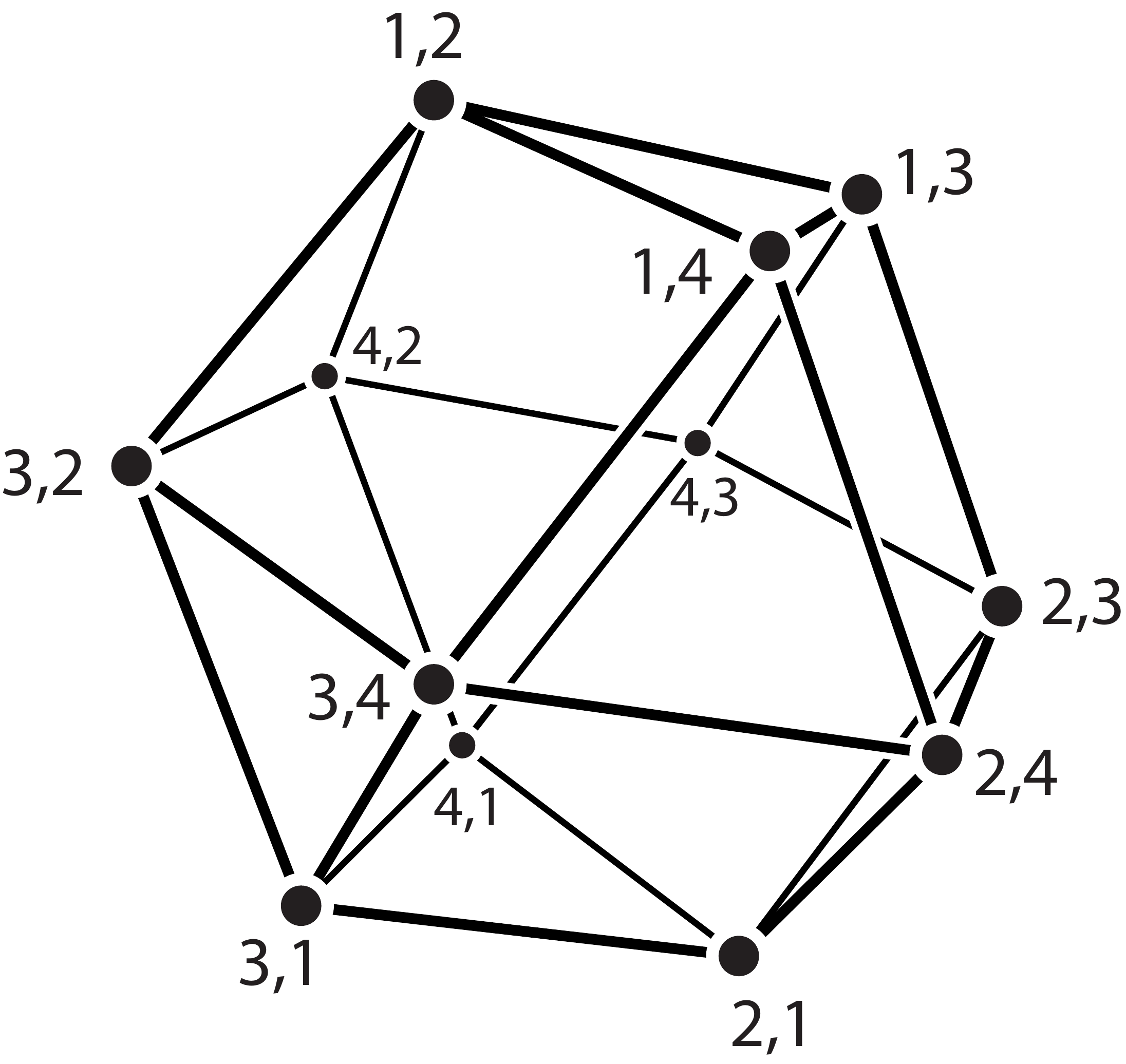}

    \caption{The polyhedral complex $\hom(K_2, K_4)$.  The vertices (which each correspond to a homomorphism $K_2 \rightarrow K_4$) are labeled by the image of each vertex of $K_2$.}
\label{Fig:homcomplex}
\end{center}
\end{figure}

We note that $\hom(K_2, G)$ has 0-cells given by the directed edges of $G$, and all faces of $\hom(K_2,G)$ are products of simplices.  The main result that we will need, originally from \cite{Lov}, is the following theorem.

\begin{theorem}\label{thm:topbound} 
For any finite graph $G$ we have 
\[\conn(\hom(K_2,G)) + 3 \leq \chi(G). \]
\end{theorem}




\begin{remark}
If a graph $G$ has multiple connected components (each with a nonzero number of edges) then both our parameters collapse, in the sense that $\zeta(G) = 2$, and $\hom(K_2,G)$ is disconnected.  At the same time, the chromatic number of $G$ is simply the maximum of the chromatic number of its components.  Hence to reasonably study Conjecture \ref{mainconj} we make the assumption that $G$ is a \emph{connected} graph.
\end{remark}

\section{2-stable families from the fundamental group} \label{sec:Fundamental}

The first case of Conjecture \ref{mainconj}, for the smallest possible values of the parameters, is not hard to establish.  As we noted in the introduction,if $G$ is a (nonempty, connected) graph we have
\[ \zeta(G) = 2 \Leftrightarrow (\conn(\hom(K_2,G)) = -1) .\]
\noindent
Indeed if $G$ is connected then $\zeta(G) = 2$ if and only if $G$ is bipartite.  At the same time a connected graph $G$ is bipartite if and only if $\hom(K_2,G)$ is disconnected (and non-empty), which by definition means that $\conn(\hom(K_2, G)) = -1$. 

Hence if $G$ is a connected graph with $\zeta(G) = 3$ we know that $G$ is not bipartite and $\conn(\hom(K_2, G)) \geq 0$.  So the first possible counterexample to Conjecture \ref{mainconj} would be a non-bipartite, connected graph $G$ satisfying $\zeta(G) = 4$ such that $\hom(K_2,G)$ has a non-trivial fundamental group, which would be implied by a nonzero first homology group.  Our first main result rules out this possibility, establishing the first nontrivial case of Conjecture \ref{mainconj} and unifying the different ways to provide the bound $\chi(G)\geq 4$.  

Our approach will involve a calculation of the first homology group $H_1(\hom(K_2,G)) = H_1(\hom(K_2,G); {\mathbb Z})$, whose generators we can describe explicitly.  For this purpose consider any even cycle \[\gamma=\cdots\dash a_{i-1}\dash b_i\dash a_i \dash b_{i+1}\dash \cdots\] of length $2n$ in $G$. We associate to $\gamma$ the $1$-chain $c(\gamma)\in C_1(\hom(K_2,G))$, defined by 
\[c(\gamma)\cq\sum_{i=0}^{n-1}\epsilon_i \left(\{a_i\} \times \{b_i, b_{i+1}\}\right) + \sum_{i=0}^{n-1} \delta_i \left(\{a_{i-1}, a_i\} \times \{b_i\}\right) ,\] where the signs $\epsilon,\delta\in\{\pm1\}$ are chosen such that \[\partial\left(\{a_i\} \times \{b_i, b_{i+1}\}\right)=\epsilon_i\cdot \left((a_i,b_{i+1})-(a_i,b_i)\right)\]  and \[\partial\left(\{a_{i-1},a_i\} \times \{b_i\}\right)=\delta_i \cdot\left((a_i,b_i)-(a_{i-1},b_i)\right).\]
It is now clear that $\partial c(\gamma)=0$, so $c(\gamma)$ is a cycle in the chain complex of $\hom(K_2,G)$. Moreover, we claim that every homology class $x\in H_1(\hom(K_2,G))$ has some representative of this form. 

\begin{lemma}
Let $x\in H_1(\hom(K_2,G))$. Then there exist even cycles $\gamma_1,\dots , \gamma_r$ in $G$ such that $x=[c(\gamma_1) +\cdots + c(\gamma_r)]$.
\end{lemma}

\begin{remark}
The cycles $\gamma_i$ are allowed to visit the same node several times, and may actually consist of two laps around an odd cycle in the graph. This  is, for example, what will happen if $G$ itself is an odd cycle.
\end{remark}

\begin{proof} 
Since the chain group $C_1(\hom(K_2,G);\mathbb{Z)})$ is generated by edges $\{a,a'\}\times\{b\}$ and $\{a\}\times\{b,b'\}$, we can write $x=[c]$, where $c\in C_1(\hom(K_2,G);\mathbb{Z})$ is a cycle with
\[\begin{aligned} c & =\sum_{i\in I_+} \left(\{a_i\} \times \{b_i, b_i'\}\right) - \sum_{i\in I_-} \left(\{a_i\} \times \{b_i, b_i'\}\right)\\ & + \sum_{j\in J_+} \left(\{a_j, a_j'\} \times \{b_j\}\right) - \sum_{j\in J_-} \left(\{a_j, a_j'\} \times \{b_j\}\right).\end{aligned}\]
We label the vertices such that \[\partial\left(\{a_i\} \times \{b_i, b_i'\}\right)=\left\{\begin{aligned}(a_i,b_i)-(a_i, b_i') \mbox{ if } i\in I_+ \\ (a_i,b_i')-(a_i, b_i) \mbox{ if } i\in I_- ,\end{aligned}\right.\] and similarily \[\partial\left(\{a_j, a_j'\} \times \{b_j\}\right)=\left\{\begin{aligned}(a_j,b_j)-(a_j', b_j) \mbox{ if } i\in I_+ \\ (a_j',b_j)-(a_j, b_j) \mbox{ if } i\in I_- .\end{aligned}\right.\] Now since $c$ is a cycle in $C_1(\hom(K_2,G);\mathbb{Z})$
\begin{equation}\label{bdry}0=\partial c=\sum_{i\in I} \left((a_i,b_i)-(a_i, b_i')\right) + \sum_{j\in J}\left((a_j,b_j)-(a_j', b_j)\right),\end{equation} where $I=I_+\cup I_-$ and $J=J_+\cup J_-$. 


Observe that, if $\partial\{a,a'\}\times\{b\}=(a',b)-(a,b)$ and $\partial\{a',a''\}\times\{b\}=(a'',b)-(a',b)$, then $\{a,a'\}\times\{b\} + \{a',a''\}\times\{b\}$ is homologous to $\{a,a''\}\times\{b\}$, as their difference is the boundary of the triangle $\{a,a',a''\}\times\{b\}$. On the other hand, if $\partial\{a,a'\}\times\{b\}=(a',b)-(a,b)$ and $\partial\{a',a''\}\times\{b\}=(a',b)-(a'',b)$, then $\{a,a'\}\times\{b\} - \{a',a''\}\times\{b\}$ is again homologous to plus or minus $\{a,a''\}\times\{b\}$, as either their difference or their sum is the boundary of the triangle $\{a,a',a''\}\times\{b\}$.

Thus, for a homology generator $c$ of minimal length, there can be no term $(a,b)$ that occurs with both signs in $\sum_{i\in I} (a_i, b_i') - \sum_{i\in I}  (a_i, b_i)$, and by the same argument, no term occurs with both signs in $\sum_{j\in J} (a_j', b_j) -  \sum_{j\in J} (a_j, b_j)$. Consequently, it follows from \eqref{bdry} that \[\sum_{i\in I} (a_i, b_i') = \sum_{j\in J} (a_j,b_j)\] and \[\sum_{i\in I} (a_i, b_i) =\sum_{j\in J} (a_j', b_j).\]

We therefore have two bijections \[\left\{\begin{aligned}\rho:I\to J & \mbox{ with } (a_i,b_i')=(a_{\rho(i)}, b_{\rho(i)}) & \mbox{ for } i\in I \\
\sigma:J\to I & \mbox{ with } (a_j',b_j)=(a_{\sigma(j)}, b_{\sigma(j)}) & \mbox{ for } j\in J, \end{aligned}\right.\] and a permutation $\phi=\sigma\circ\rho:I\to I$. By construction, $(a_i, b_i)$ and $(a_i,b_{\phi(i)})=(a_{\rho(i)}b_{\sigma(\rho(i))})=(a_{\rho(i)}b_{\rho(i)})$ are edges in $G$ for all $i\in I$.

We write $\phi=\phi_1\cdots \phi_r$ as a product of $r$ cyclic permutations on disjoint ground sets. Let $i\in I$ be an arbitrary representative of the cycle $\phi_s$, which has length $\ell=\ell_s$, $s=1,\dots, r$. Then $G$ contains a cycle \[a_i\dash b_{\phi(i)}\dash a_{\phi(i)}\dash b_{\phi^2(i)}\dash\cdots\dash a_{\phi^{\ell}(i)}=a_i,\] which we will denote by $\gamma_s$.  
Now, we introduce the signs $\epsilon_i$ and $\delta_i$ such that  \[\partial\left(\{a_i\} \times \{b_i, b_{\phi(i)}\}\right)=\epsilon_i\cdot \left((a_i,b_{\phi(i)})-(a_i,b_i)\right)\]  and \[\partial\left(\{a_{i},a_{\phi(i)}\} \times \{b_i\}\right)=\delta_i \cdot\left((a_{\phi(i)},b_i)-(a_i,b_i)\right).\]
We then have \[\begin{aligned} c(\gamma_1) + \cdots + c(\gamma_r) &= \sum_{i\in I}\epsilon^i(\{a_i, a_{\phi(i)}\}\times \{b_{\phi(i)}\}) + \sum_{i\in I}\delta^i(\{a_i\}\times \{b_i, b_{\phi(i)}\})
\\ & = \sum_{i\in I}\epsilon^i(\{a_{\rho(i)}, a_{\rho(i)}'\}\times \{b_{\rho(i)}\}) + \sum_{i\in I}\delta^i(\{a_i\}\times \{b_i, b_i'\})
\\ & =\sum_{i\in I_+} \left(\{a_i\} \times \{b_i, b_i'\}\right) - \sum_{i\in I_-} \left(\{a_i\} \times \{b_i, b_i'\}\right)\\ & + \sum_{j\in J_+} \left(\{a_j, a_j'\} \times \{b_j\}\right) - \sum_{j\in J_-} \left(\{a_j, a_j'\} \times \{b_j\}\right)
\\& = c, \end{aligned}\]
so $x=[c]=[ c(\gamma_1) + \cdots + c(\gamma_r)]$.


\end{proof}

Denote by $\ell(\gamma)$ the length of the cycle $\gamma$. Also, if $\pi=\pi_1\cdots\pi_a$ and $\pi'=\pi'_1\cdots\pi'_b$ are paths in $G$, where $\pi_a=\pi'_1$, we let $\pi\cdot \pi'$ denote their concatenation $\pi\cdot \pi'=\pi_1\cdots\pi_{a-1}\pi'_1\pi'_2\cdots\pi'_b$.
We are now ready to state our first theorem.

\begin{theorem}\label{thm:fundgroup}
Suppose $G$ is a graph such that $H_1(\hom(K_2,G))$ contains an infinite cyclic subgroup. Then $\zeta(G)\leq 3$.
\end{theorem}

\begin{proof}
Suppose $x\in H_1(\hom(K_2,G))$ generates a subgroup isomorphic to $\Z$ and assume, replacing $x$ by one of its summands if necessary, that $x=c(\gamma)$ for some even cycle $\gamma$. We need to show that there exists a $2$-stable family in $G$. We will construct such a family $\{A_i, B_i\}_{i\in \mathbb{Z}_n}$, for some $n$ defined below, as follows. First consider the set of all even cycles $\gamma$ in $G$ such that $c(\gamma)=r_{\gamma}x$ for some $r_{\gamma}\in\Z_{>0}$. Then, choose from this set a cycle that minimizes $\ell(\gamma)/r_{\gamma}$.
Let $n=\ell(\gamma)/2$, and write 
\[\gamma =\cdots\dash a_{i-1}\dash b_i\dash a_i \dash b_{i+1}\dash \cdots,\]
where indices are taken modulo $n$. Note that $c(\gamma)$ has infinite order in $H_1(\hom(K_2,G))$, so we may assume (possibly redefining $x$) that $x=c(\gamma)$ and $r_\gamma=1$.  We begin our construction of the sets $A_i$ and $B_i$ by declaring  $a_i\in A_i$ and $b_i\in B_i$ for each $i=0,1, \dots ,n-1$. Now recursively, if there exists a cycle 
\[\gamma' =\cdots\dash c_{i-1}\dash d_i\dash c_i \dash d_{i+1}\dash \cdots\] of length $2r n$ in $G$, such that $c_k=a_k$ or $d_k=b_k$ for some $k$ and with $r c(\gamma)=c(\gamma')$, then we add $c_{i+2sn}\in A_i$ and $d_{i+2sn}\in B_i$, for every $s\in \Z$. We will say that the cycle $\gamma'$ \emph{forces} $c_{i+2sn}\in A_i$ and $d_{i+2sn}\in B_i$.

For the completion of the proof we need to show that $A_i$ and $B_i$ are proper subsets of $V(G)$ and that $N(A_i)\cap N(A_{i+1})= B_{i+1}$. By the same argument we will then get $N(B_i)\cap N(B_{i-1})= A_{i-1}$. 


To show that $A_i$ and $B_i$ are proper subsets of $V(G)$, by symmetry (rotations of the cycle $\gamma$) it suffices to show that $A_0\neq V(G)$. We will therefore show by contradiction that $a_1\not\in A_0$. If we had $a_1\in A_0$, there would be two paths $\pi$ from $a_0$ to $a_1$ and $\pi'$ from $a_1$ to $a_0$ in $G$, with $\ell(\pi)=2rn$, $\ell(\pi')=2r'n$, and with $c(\pi\cdot\pi')$ a representative of $(r+r')x$ in $H_1(\hom(K_2,G))$. Then we have two cycles \[\tilde\pi \cq \pi \cdot (a_1\dash b_2\dash\cdots\dash a_0)\] and \[\tilde\pi'\cq (a_0\dash b_1\dash a_1)\cdot\pi,\] with $\ell(\tilde\pi) + \ell(\tilde\pi')=2(r+r'+1)n$ and  $c(\tilde\pi) + c(\tilde\pi')=2(r+r'+1)x$. Since none of the cycles $\tilde\pi$ or $\tilde\pi'$ have lengths that are multiples of $2n$, one of them has a smaller fraction $\ell/r$ than that of $\gamma$, which contradicts the minimality according to which $\gamma$ was chosen.

We need to prove that $N(A_i)\cap N(A_{i+1})= B_{i+1}$. The inclusion $B_{i+1}\subseteq N(A_i)\cap N(A_{i+1})$ follows directly from the construction. Assume $v\in N(A_i)\cap N(A_{i+1})$, meaning that there is a path $c_i\dash v\dash c'_{i+1}$ with $c_i\in A_i$ and $c'_{i+1}\in A_{i+1}$. Let 
\[a_k\dash d_{k+1}\dash\cdots\dash d_i\dash c_i \dash d_{i+1}\dash \cdots\dash d_k\dash a_k\]
 be a cycle forcing $c_i\in A_i$, with \[\ell(a_k\dash\cdots\dash c_i)=2rn + 2(i-k)\] and \[\ell(c_i \dash d_{i+1}\dash \cdots\dash d_k\dash a_k)=2sn-2(i-k).\]
Analogously, let 
\[a_{k'}\dash d_{k'+1}\dash\cdots d'_{i+1}\dash c'_{i+1} \dash d'_{i+2}\dash \cdots\dash d_{k'}\dash a_{k'}\]
 be a cycle forcing $c'_{i+1}\in A_{i+1}$, with \[\ell(c'_{i+1}\dash\cdots\dash a_{k'})=2r'n + 2(k'-i-1)\] and \[\ell(a_{k'}\dash\cdots\dash c'_{i+1})=2s'n - 2(k'-i-1).\] Then these cycles, together with $\gamma$, can be concatenated to form cycles
\[a_k\dash d_{k+1}\dash\cdots\dash c_i\dash v\dash c'_{i+1} \dash \cdots\dash d'_{k'}\dash a_{k'}\dash\cdots\dash a_k\] 
and
\[a_k'\dash d'_{k+1'}\dash\cdots\dash c_i'\dash v\dash c_{i+1} \dash \cdots\dash d_{k}\dash a_{k}\dash\cdots\dash a_k'.\]
The sum of these cycles generate $(r+r'+s+s')x$ in $H_1(\hom(K_2,G))$, and the sum of their lengths is $2(r+r'+s+s')n$. Hence they both have $\ell/r=\ell_\gamma/r_\gamma$, as this is the minimal possible value of $\ell/r$ by construction. Thus we have $v\in B_{i+1}$, which finishes the proof.
\end{proof}

It is likely that the criterion that $x$ has infinite order is superfluous, so that we have $\zeta(G)\leq 3$ whenever $H_1(\hom(K_2,G))\neq 0$. The proof of Theorem \ref{thm:fundgroup} relies on a homology generator minimizing the quotient $\ell/r$. It is clear that such a generator must exist for finite graphs, but there is no reason to expect that Theorem \ref{thm:fundgroup} would fail for infinite graphs.

\section{Warmth of twisted products} \label{sec:Twisted}

In \cite{DocSch}, a family of \emph{twisted toroidal graphs} $\{T_{k,m}: k,m\geq 1\}$ is constructed with the property that the connectivity of their neighborhood complex is an arbitrarily bad lower bound for the index of the same space, and hence also for the chromatic number.  These graphs are natural candidates for a counterexample to Conjecture \ref{mainconj} since the topological properties of $\hom(K_2, T_{k,m})$ that provide a lower bound on $\chi(G)$ are not detected by connectivity (they are examples of so-called \emph{non-tidy} spaces, see \cite{DocSch} for details).

In this section we show that in fact these graphs also have small warmth, thus presenting a new family of graphs for which the warmth bound is arbitrarily far from the chromatic number, and also further supporting Conjecture~\ref{mainconj}.  We will first settle the definitions, which agree with those in \cite{DocSch}.

\begin{definition}
Let $\Gamma$ be a group, and suppose $G$ and $H$ are graphs with left $\Gamma$-actions.  The \emph{twisted product} $G \times_{\Gamma} H$ is the graph with vertices and edges given by
\[V(G\times_{\Gamma} H)\cq V(G)\times V(H)/\sim,\] where $(\gamma g, h)\sim (g,\gamma h)$ for every $g\in V(G), h\in V(H), \gamma\in\Gamma$, and
\[E(G\times_{\Gamma} H)\cq \{\{(g,h) , (g',h')\} : gg'\in E(G), hh'\in E(H)\}\]
 \end{definition}

In particular, if $\Gamma$ acts trivially on both $G$ and $H$, then the twisted product is just the \emph{categorical} or \emph{direct} product. It is easy to check that twisted products are associative, so that \[(A\times_\Gamma B)\times_\Gamma C\cong A\times_\Gamma(B \times_\Gamma C).\]  We then have the following definition.

\begin{definition}
Let $C_{2m}^\circ $ be the cycle graph of even length $2m$ with a loop attached at every vertex, endowed with the antipodal $\Z_2$-action. Let $K_2$ have the $\Z_2$-action interchanging its two vertices. Define the \emph{twisted toroidal graph} $T_{k,m}$ recursively by $T_{0,m}=K_2$ for all $m\geq 1$, and $T_{k+1, m}\cq T_{k,m}\times_{\Z_2}C_{2m}^\circ$.
\end{definition}

A concrete representation of $T_{k,m}$ is thus as follows. Its vertices are equivalence classes of $(k+1)$-tuples $(\epsilon, a_1,\dots , a_k)$, where $\epsilon\in\{\pm\}$ and $a_i\in\Z$ (or $\Z/(2m)$), modulo the identifications \[(\epsilon, a_1,\dots, a_k)=(-\epsilon, a_1,\dots , a_{i-1}, a_i + m, a_{i+1},\dots ,a_k).\] The neighborhoods are given by \[N([\epsilon, a_1,\dots , a_k])=\{[- \epsilon, a_1+\alpha_1,\dots , a_k + \alpha_k] : \alpha_i\in\{-1,0,1\}\}.\] In particular, $T_{k,m}$ has $2m^k$ vertices, and is $3^k$-regular.  See Figure \ref{Fig:toroidal} for a depiction of $T_{1,3}$.

\medskip

\begin{figure}[ht]
\begin{center}
  \includegraphics[scale = 0.3]{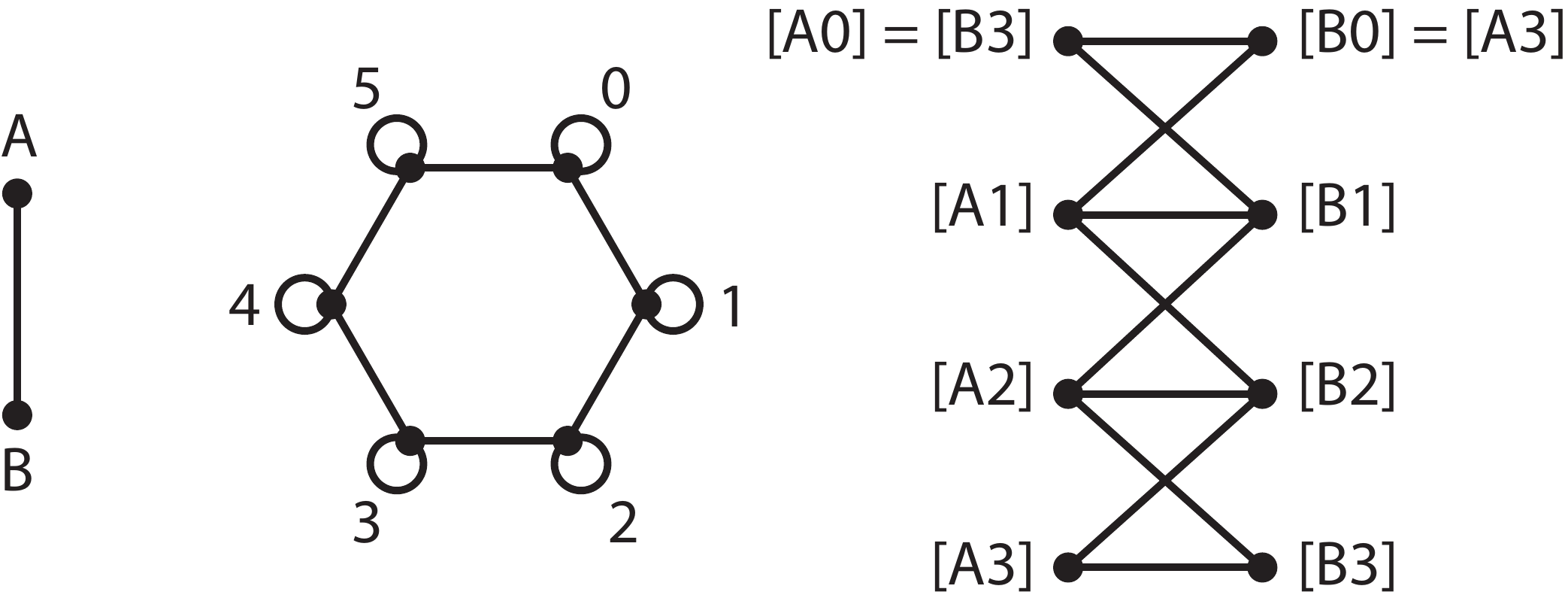}

    \caption{The graphs $K_2$, $C_6^\circ$, and $T_{1,3} = K_2 \times_{\Z_2} C_6^\circ$.}
\label{Fig:toroidal}
\end{center}
\end{figure}

We are now ready to prove that $\zeta(T_{k,m})\leq 3$ for large enough $m$.  We will establish this as a consequence of a more general result regarding twisted products.

\begin{lemma}\label{product}
Let $\Gamma$ be a group that acts on a pair of graphs $G$ and $H$. Let $\zeta(G)\leq d+1$, and suppose $\mathcal A$ is a $d$-stable family in $G$ that is invariant under the $\Gamma$-action. Then $\zeta(G\times_\Gamma H)\leq d+1$.
\end{lemma}
\begin{proof}
We construct the desired collection $\mathcal B$ of subsets of $V(G\times_\Gamma H)$ by defining
\[\mathcal{B}\cq\{\{[g,h]: g\in A\} : A\in \mathcal{A}\}.\] 
We claim that $\mathcal{B}$ is a $d$-stable family. Fix an element $B_A=\{[g,h]: g\in A\}$ of $\mathcal B$. Let $\{A_i\}_{i=1}^d\subseteq\mathcal{A}$ be such that $\cap_i N(A_i)= A$; such a subfamily exists since $\mathcal A$ is $d$-stable. We claim that $\cap_i N(B_{A_i})= B_A$. Indeed, assume $[g,h]\in \cap_i N(B_{A_i})$ is adjacent to some $[g_i, h_i]\in B_{A_i}$ for every $i=1,\dots, d$. By construction, we can choose representatives $(g_i,h_i)$ for the equivalence classes $[g_i, h_i]$ such that $g g_i$ is an edge in $G$ and $hh_i$ is an edge in $H$ for every $i$. Moreover, we have $g_i\in A_i$ by construction (since $\mathcal A$ is $\Gamma$-invariant). Now $g\in \cap_i N(A_i)= A$, so $[g,h]\in B_A$. This proves the lemma.
\end{proof}

Just as quotient maps modulo group actions are local homeomorphisms only if the group action is properly discontinuous, we need a notion of discontinuity to guarantee that combinatorial properties of graphs remain valid modulo the group action. This motivates the following definition, where for vertices $v, w \in G$ we let $d(v,w)$ denote the graph distance (length of the shortest path between $v$ and $w$).

\begin{definition}
Let $\Gamma$ be a group acting on a graph $G$. If for every $v\in V(G)$ and for every non-identity element $\gamma \in \Gamma$, we have $d(v,\gamma v)\geq D$, then $\Gamma$ is said to act $D$-discontinuously.
\end{definition}

\begin{lemma}\label{stable}
Suppose that $G$ has a $d$-stable family consisting of singletons, and that $\Gamma$ acts $5$-discontinuously on $G$. Then  $G$ has a $d$-stable family  that is invariant under the $\Gamma$-action.
\end{lemma}
\begin{proof}
Let ${\mathcal B}$ denote the given $d$-stable family.  We will abuse notation slightly and think of ${\mathcal B}$ as a subset of $V(G)$ (as opposed to a collection of singleton sets).  Next define
\[{\mathcal A} \cq \{\Gamma(x) : x \in {\mathcal B}\},\] where $\Gamma(x) = \{\gamma x : \gamma \in \Gamma\}$ denotes the orbit of the element $x$. For $x \in {\mathcal B} $, by assumption there exist  $x_1, x_2, \dots, x_d$ in ${\mathcal B}$ such that $\cap_i N(x_i)=\{x\}$. Clearly $\gamma x\in\cap_i N(\Gamma(x_i))$ for every $\gamma\in\Gamma$. 

On the other hand, assume $y\in\cap_i N(\Gamma\{x_i\})$, so that $y$ is adjacent to $\gamma_i x_i$ for some $\gamma_i \in \Gamma$, for all $i=1,\dots, d$. We claim that all the $\gamma_i$ are in fact equal.  If not we would have a four-path \[\gamma_i x\dash \gamma_i x_i\dash y\dash\gamma_j x_j\dash \gamma_j x\] between different elements in the $\Gamma$-orbit of $x$, contradicting the fact that $\Gamma$ acts $5$-discontinuously. So then we have 
\[y\in \cap_i N(\gamma_i x_i)=\{\gamma_i x\}.\]
Hence in particular we have $y\in \Gamma(x)$. We conclude that $\mathcal A= \{\Gamma\{x\} : x\in V(G)\}$ is a $d$-stable family, and is $\Gamma$-invariant by construction.
\end{proof}

The previous two lemmas allow us to determine the warmth of $T_{k,m}$:

\begin{theorem}\label{thm:twisted}
For all $m\geq5$ and $k \geq 1$ we have $\zeta(T_{k,m}) = 3$.
\end{theorem}
\begin{proof}
For $2m > 4$ the looped cycle graph $C_{2m}^\circ$ has a $2$-stable family consisting of singletons, as $N(i-1)\cap N(i+1)=\{i\}$.
The group $\Z_2$ acts $5$-discontinuously on $C_{2m}^\circ$ when $m\geq 5$, so by Lemma~\ref{stable}, $C_{2m}^\circ$ has a $\Gamma$-invariant $2$-stable family. Thus by Lemma~\ref{product}, \[\zeta(T_{k+1,m})=\zeta(T_{k,m}\times_{\Z_2}C_{2m}^\circ)\leq 3,\] when $k\geq 0$. On the other hand, $T_{k+1,m}$ is not bipartite, so it has warmth $\zeta(T_{k+1,m})>2$. This concludes the proof.
\end{proof}

\section{Bipartite subgraphs}\label{sec:Bipartite}

In this section we investigate the effect that local structure (in terms of subgraph containment) has on the warmth of a graph $G$.  In particular, in support of Conjecture \ref{mainconj}, we derive results analogous to those known from previous work for neighborhood complexes.

Recall that the \emph{(complete) bipartite graph} $K_{A,B}$ is a graph with vertex set $V(K_{A,B}) = A \coprod B$ and edges given by all pairs $\{(v,w): v\in A, w \in B\}$.  For finite graphs with $|A| = a$, $|B| = b$, we write $K_{a,b}$ to denote the graph.  In his original paper on the neighborhood complex, Lov\'asz established the following result.

\begin{lemma}
If a graph $G$ does not contain the bipartite graph $K_{a,b}$ for some $a+b = n$, then $N(G)$ deformation retracts onto a complex of dimension $n-3$.  
In particular, if $\chi(G)$ is finite then the connectivity of $N(G)$ is no more than $n-4$.
\end{lemma}

We can ask about the influence of bipartite subgraphs on the warmth of a graph.  Recall that $N(G)$ is homotopy equivalent to $\hom(K_2,G)$, so if Conjecture \ref{mainconj} is true, it then follows that if $G$ does not contain the complete bipartite graph $K_{a,b}$ then necessarily $\zeta(G) \leq a + b - 1$.  Our next main result says that this is indeed the case, providing further evidence for Conjecture \ref{mainconj}.  This generalizes the result that a graph $G$ with girth $g\geq 5$, always has $\zeta(G)\leq 3$. Indeed, graphs with girth $g\geq 5$ do not contain any copy of $C_4\cong K_{2,2}$ as a subgraph.

We will need the following notion.

\begin{definition}
Suppose $G$ is a graph with vertices $v$ and $w$ such that $N(v) \subseteq N(w)$.  Then the graph homomorphism that sends $v$ to $w$ and every other vertex of $G$ to itself is a retraction of $G$ onto the graph $G \backslash \{v\}$. We call this map a \emph{fold} (or a \emph{folding}) and denote it $f_{vw}$.  A graph $G$ is said to be \emph{stiff} if no foldings are available.
\end{definition}

\begin{figure}[ht]
\begin{center}
  \includegraphics[scale = 0.35]{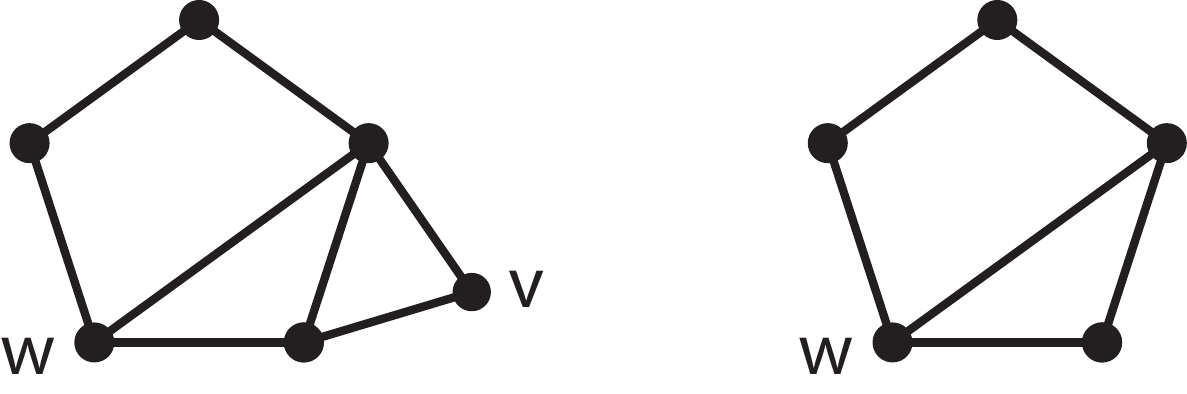}

    \caption{A graph $G$ and the folded $G \backslash \{v\}$.}
\label{Fig:fold}
\end{center}
\end{figure}

See Figure \ref{Fig:fold} for an illustration of a folding.  We then have the following result from \cite{BriWin}. 

\begin{theorem}
If $f_{vw}: G \rightarrow G \backslash \{v\}$ is a folding then we have
\[\zeta(G) = \zeta(G \backslash \{v\}).\]
\end{theorem}

We note that foldings also preserve the homotopy type of the neighborhood complex, and in fact of arbitrary $\hom$ complexes, as shown in \cite{Doc} and \cite{Kozshort}.  Indeed, if $f_{vw}: G \rightarrow G \backslash \{v\}$ is a folding then there is an induced homotopy equivalence
\[(f_{vw})_* : N(G) \rightarrow N(G \backslash \{v\}).\]

Hence in our consideration of warmth $\zeta(G)$ and connectivity of $N(G)$ we may always assume that the graph $G$ is stiff. We will say that $v\in V(G)$ is \emph{generated} by $\{u_1,\dots , u_r\}$ if $\{v\}=\cup_{i=1}^r N(u_i)$.

\begin{lemma}\label{lm:bipartite}
Let $G$ be a stiff graph, and suppose $a$ and $b$ are positive integers with $a+b=n$. Assume that $v\in V(G)$ is not generated by any $n-2$ or fewer of its neighbors, i.e., that there is no collection of vertices $\{u_1,\dots ,u_{k}\}$ with $k \leq n-2$ such that $\cap_{i=1}^{k}N(u_i)=\{v\}$. 
Then there exist $v_2,\dots,v_a, w_1,\dots ,w_b$ such that $K_{(\{v,v_2,\dots v_a\},\{w_1,\dots,w_b\})}\subseteq G$.
\end{lemma}
\begin{proof}
We will fix $n$, and proceed by induction on $a$. Let $v_1$ be a vertex not generated by $n-2$ of its neighbors. By stiffness, $v_1$ is generated by its set of neighbors, so by assumption $v_1$ has to have at least $n-1$ neighbors. This proves the case $a=1$.

Assume, inductively, that $a>1$, and that we have found a complete bipartite subgraph \[K_{(\{v_1,\dots v_{a-1}\},\{w_1,\dots,w_{b+1}\})}\subseteq G.\] 
For $i=2,\dots, a-1$, by stiffness we can find $u_i\in N(v_1)\smallsetminus N(v_i)$. By assumption, $\{v_1\}$ is a proper subset of \[\bigcap_{i=2}^{a-1}N(u_i)\cap\bigcap_{j=1}^{b}N(w_i),\] so let $v_a$ be another member of this set. By construction, $v_a\not\in\{v_1,\dots ,v_{a-1}\}$. Now $\{w_1,\dots,w_{b}\}\subseteq N(v_a)$, so \[K_{(\{v\dots v_{a}\},\{w_1,\dots,w_{b}\})}\subseteq G.\]
\end{proof}

\begin{corollary}\label{corr:bipartite}
Suppose $a$ and $b$ positive integers with $a+b \geq 3$.  If a graph $G$ does not contain any subgraph isomorphic to $K_{a,b}$, then $\zeta(G)\leq a+b-1$.
\end{corollary}
\begin{proof}
Let $v\in V(G)$ be an arbitrary node. By assumption, there are no \[v_2,\dots ,v_a,w_1,\dots,w_b\] such that $K_{(\{v\dots v_{a}\},\{w_1,\dots,w_{b}\})}\subseteq G$, so by (the contrapositive of) Lemma~\ref{lm:bipartite}, there is $\{u_1,\dots ,u_{a+b-2}\}$ with $\cap_{i=1}^{a+b-2}N(u_i)=\{v_i\}$. This means that the singletons \[\{\{v\}:v\in V(G)\}\] is an $(a+b-2)$-stable family, so $\zeta(G)\geq a+b-1$.
\end{proof}

\section{Warmth and Connectivity of Random Graphs}\label{random}
The warmth of random graphs has been studied in~\cite{FadKah, Sha}, while neighborhood complexes of random graphs have been considered in~\cite{Kah}. In this context the objects of study are drawn from the Erd\H{o}s-Renyi random graph model $G(n,p)$, where one considers simple graphs with nodes labeled $1,\dots , n$, containing the edge $ij$ (where $i\neq j$) with probability $p$, independently for each pair $i,j\in [n]$.  We point out that Conjecture~\ref{mainconj} is consistent with the known results on random graphs from~\cite{FadKah,Kah}. 




In \cite{Sha}, the results on the warmth of random graphs are adapted to a somewhat more general model of random graphs. We now fix the expected degrees $w_i$ of each vertex $i$, such that $0\leq w_i\leq n-1$ and $w_i^2\leq\sum_{k=1}^n w_k$ for each $i\in[n]$. Then we include the edge~$ij$ with probability $p_{ij}=\frac{w_i w_j}{\sum_{k=1}^n w_k}$, again independently of all other edges. It is shown (Theorem 2 in \cite{Sha}) that in this setting, the bound \[(1-\delta)\log(n)\leq \zeta(G)\] remains valid in the dense regime, where $\min_i w_i^2=\Theta\left(\sum_{k=1}^n w_k\right)$. 

The lower bounds from~\cite{Kah} are obtained by bounding the probability that the neighborhood complex is $k$-neighborly, meaning that every $k$-tuple of vertices have a neighbor in common. This technique extends immediately to the non-homogeneous random graph model in \cite{Sha}, to give the same lower bound of \[(1-\delta)\log(n)\leq\conn(\hom(K_2,G(n,p)))\] as in the homogeneous case. Putting this together, we conclude that the known results for random graphs are consistent with Conjecture~\ref{mainconj}

\section{Further Questions}\label{sec:further}
Our work leaves open a number of open questions, the most obvious being the remaining cases of Conjecture \ref{mainconj}. As we have mentioned, even our proof of the case addressed in Theorem \ref{thm:fundgroup} requires a condition on the first homology group that we should be able to remove.

We also point out that even if Conjecture \ref{mainconj} turns out to be false, there could still be an interesting connection between warmth of a graph and the topology of its neighborhood complex via the various notions of ${\mathbb Z}_2$-indices of $\hom(K_2,G)$.  In the context of lower bounding $\chi(G)$ it turns out that there are two other natural topological invariants to consider, namely the \emph{index} and \emph{coindex} of the space $\hom(K_2,G)$. For a ${\mathbb Z}_2$-space $X$ the index and coindex are defined, respectively, as 
\[\ind(X) = \min\{j: \textrm{There exists a ${\mathbb Z}_2$-equivariant map $X \rightarrow {\mathbb S}^j$}\}.\]
\[\coind(X) = \max\{k: \textrm{There exists a ${\mathbb Z}_2$-equivariant map ${\mathbb S}^k \rightarrow X$}\}.\]

\noindent
Here ${\mathbb S}^j$ is considered a free ${\mathbb Z}_2$-space with the antipodal action.  The more precise version of the original result from \cite{Lov} is then that $\chi(G) \geq \ind(\hom(K_2,G))+2$, for any graph $G$.  It can be shown that for a ${\mathbb Z}_2$-space $X$ we have the inequalities
\[\conn(X) + 1 \leq \coind(X) \leq \ind(X).\]
\noindent
Hence the weakest version of Conjecture \ref{mainconj} that still captures the desired implication regarding warmth and topology of the Hom complex is given by the following.

\begin{conj}\label{mainconjweak}
For any finite graph $G$ we have $\zeta(G) \leq \ind(\hom(K_2,G)) + 2$.
\end{conj}

It may well be the case that Conjecture \ref{mainconjweak} is easier to prove, as it in particular avoids homotopy groups and speaks directly to the existence of equivariant maps into spheres.   By considering Stiefel-Whitney classes one can also get a lower bound the index in terms of vanishing of ${\mathbb Z}_2$ homology groups of $\hom(K_2, G)$ (see \cite{Koz}), a seemingly more combinatorial condition. In \cite{DocSch} the index and coindex of $\hom(K_2, G)$ is characterized graph theoretically, the latter in terms of maps from certain `spherical graphs', and the former in terms of the chromatic number of a family of graphs obtained from $G$.  As we have seen, if $G$ is a connected graph with $\zeta(G) = 3$ then $\conn(\hom(K_2,G)) = 0$ and hence the revised conjecture \ref{mainconjweak} holds.  We have established a version of the main conjecture for the case of $\zeta(G) = 4$.  It seems possible that results from the recent preprint \cite{SimTarWeh} might be employed to establish $\coind(\hom(K_2, G)) \geq 2$ directly, although we have yet to find a connection.

Also, in the spirit of Section~\ref{sec:Bipartite}, it may be interesting to study how avoidance of other subgraphs can provide upper bounds on  the connectivity and warmth. Some topological results in this direction can be found in~\cite{Kah}. Finally, following ideas from Section~\ref{random}, there are many variations of random graph models for which nothing or very little is known about both warmth and topology. A natural model to consider in this regard is that of preferential attachment.


\begin{thebibliography}{30}

\bibitem{BabKoz}
E.~Babson, D.~Kozlov, {\em Complexes of graph homomorphisms}, Israel Journal of Mathematics {\bf 152} (2006), 285--312.

\bibitem{BabKozproof}
E.~Babson, D.~Kozlov, {\em Proof of the Lov\'asz conjecture}, Annals of Mathematics, {\bf 165} (2007), 965--1007.

\bibitem{BriWin}
G.~Brightwell, P.~Winkler, {\em Graph Homomorphisms and Long Range Action}, CDAM Research Report LSE-CDAM-2001-07 (2001), 20pp.

\bibitem{CLSW}
P.~Csorba, C.~Lange, I.~Schurr, A.~Wassmer, {\em Box complexes, neighborhood complexes, and the chromatic number},
J. Combin. Theory Ser. A {\bf 108} (2004) 159--168.

\bibitem{Doc}
A.~Dochtermann, {\em Hom complexes and homotopy theory in the category of graphs}, European Journal of Combinatorics {\bf 30} (2009), 490-509.

\bibitem{DocSch}
A.~Dochtermann, C.~Schultz, {\em Topology of Hom complexes and test graphs for bounding chromatic number}, Israel Journal of Mathematics {\bf 187} (2012), 371--417.

\bibitem{FadKah}
S.~Fadnavis, M.~Kahle, {\em Warmth and mobility of random graphs}, {\tt http://arxiv.org/abs/math/1009.0792} (2010), 13pp.

\bibitem{GyaJenSti}
A.~Gy\'arf\'as, T.~Jensen, M.~Stiebitz, {\em On graphs with strongly independent color-classes}, Journal of Graph Theory {\bf 46}  (2004), 1--14.

\bibitem{Kah}
M.~Kahle, {\em The Neighborhood Complex of a Random Graph}, Journal of Combinatorial Theory, Series A {\bf 114}  (2007), 380--387.

\bibitem{Koz}
D.~Kozlov, {\em Homology tests for graph colourings}, Contemporary Mathematics {\bf 423}  (2006), 221--234.

\bibitem{Kozshort}
D.~Kozlov, {\em A simple proof for folds on both sides in complexes of graph homomorphisms}, Proceedings of the American Math Society {\bf 134} (2006), 1265--1270.

\bibitem{Lov} 
L.~Lov\'asz, {\em Kneser's conjecture, chromatic number, and homotopy}, Journal of Combinatorial Theory, Series A {\bf 25} (1978), 319--324.


\bibitem{Sha} 
Y.~Shang, {\em A Note on the Warmth of Random Graphs with Given Expected Degrees}, International Journal of Mathematics and Mathematical Sciences {\bf 2014} (2014), A749856, 4pp.

\bibitem{SimTarWeh}
G.~Simons, C.~Tardif, D.~Wehlau, {\em Generalized Mycielski graphs and bounds on chromatic numbers}, {\tt http://arxiv.org/abs/math/1601.04642} (2016), 14pp.
\end{thebibliography}
\end{document}